\documentclass[preprint,11pt]{elsarticle}
\usepackage[margin=1.5in, footskip=.5in]{geometry}

\usepackage{amssymb, amsmath, amsfonts}
\usepackage{mathrsfs}
\usepackage{multicol}
\usepackage{amsthm}

\usepackage{graphicx}
\usepackage{amssymb}
\usepackage{relsize}

\usepackage{lineno}

\usepackage{xcolor}

\newtheorem{Prop}{Proposition}
\newtheorem{Th}{Theorem}
\newtheorem{Cor}{Corollary}
\theoremstyle{remark}
\newtheorem{Rem}{Remark}
\newtheorem{Ex}{Example}

\newcommand{\p}{{\sf P}}
\newcommand{\pF}{{\sf P}_F}
\newcommand{\E}{{\sf E}}
\newcommand{\EF}{{\sf E}_F}

\newcommand{\F}{{\sf F}}

\newcommand{\N}{{\mathbb N}}
\newcommand{\R}{{\mathbb R}}
\newcommand{\f}{{\mathfrak F}}
\newcommand{\B}{{\mathcal B}}

\newcommand{\Var}{{\sf var}}

\newcommand{\1}{{\sf 1}}

\newcommand{\note}[1]{\textcolor{black}{#1}}

\journal{}

\begin{document}

\begin{frontmatter}


\title{The limit theorem for maximum of partial sums of exchangeable random variables}



\author[a1]{Patricia Alonso Ruiz}
\author[a2]{Alexander Rakitko}

\address[a1]{University of Connecticut}
\address[a2]{Moscow State University}

\begin{abstract}
\footnotesize We obtain the analogue of the classical result by Erd\"os and Kac on the limiting distribution of the maximum of partial sums for exchangeable random variables with zero mean and variance one. We show that, if the conditions of the central limit theorem of Blum et al. hold, the limit coincides with the classical one. Under more general assumptions, the probability of the random variables having conditional negative drift appears in the limit.
\end{abstract}


\end{frontmatter}


\section{Introduction}\label{intro}
Erd\"os and Kac established in~\cite{EK46} some fundamental results on the distribution of the maximum of partial sums $S_k:=\sum_{i=1}^kX_i$, where $\{X_n\}_{n\in\N}$ is a sequence of \note{independent, identically distributed} (i.i.d.) centered random variables \note{with variance one}. In particular, they proved that the limiting distribution of $n^{-\frac{1}{2}}\max_{1\leq k\leq n}S_k$ is given by \note{$(2\Phi(x)-1)\1_{[0,\infty)}(x)$}, where $\Phi(\cdot)$ denotes the \note{probability} distribution function \note{(p.d.f.)} of the standard normal distribution. 

\medskip

\note{Our interest in studying the (rescaled) maximum of partial sums is motivated by its manifold applications. On the one hand, it is directly related to first passage times of random walks and renewal theory~\cite{Hey67,Sie68}. On the other hand, in the classical i.i.d. setting, this statistic has since long been employed in numerous research areas such as hydrology~\cite{Bui82}, reservoir storage~\cite{Hur51} and change-point analysis~\cite{JJS87}. Moreover, as a matter of study in extreme value theory, this type of limit theorems are of especial relevance, for instance in finance (see~\cite{Nov12} and references therein).}

\medskip

The purpose of this paper is to generalize \note{the original result of Erd\"os and Kac} to exchangeable sequences of random variables \note{and thereby extend the mentioned statistic to further stochastic models. Exchangeable random variables, introduced by de Finetti in~\cite{DeF30}, are random variables with the property of being} conditionally independent. \note{Equivalently, one can think of them as mixtures of i.i.d. random variables directed by a random measure.} The study of classical results of probability theory in the exchangeable setting started with the Central Limit Theorem (CLT) by Blum, Chernoff, Rosenblatt and Teicher in~\cite{BCRT58} and it led to a series of works~\cite{Web80,Eag82,HT87} that continues expanding (see e.g.~\cite{BR97,Cha06,FLR12,YL15}). \note{E}xchangeable random variables \note{are} of great interest due to their versatility as stochastic models~\cite{Aus08,PR09,Ald10} and their wide applicability in genetics~\cite{Kin80}, Bayesian analysis~\cite{CM15} and \note{many other branches} of statistical analysis~\cite{GST00,YZ08,BR15}.

\medskip

\note{The limit theorem considered in this paper contributes to extend results of extreme value theory to the exchangeable context.} In this direction, Berman obtained in~\cite{Ber62} the limiting distribution of the maximum of an exchangeable sequence of random variables. 

\medskip

Our results show that, if the classical conditions of the CLT of Blum et al. hold, one obtains the original statement of Erd\"os and Kac in the i.i.d. setting, c.f. Proposition~\ref{Th EK ex case}. Dropping off the assumption on the 
\note{variance of the directing random measure} gives rise in \note{Theorem}~\ref{Th EK weak} to a limiting distribution that, \note{in the non-degenerate case,} resembles the previous result \note{and} involves the distribution function of a mixture of Gaussians. \note{Consequently, we discover in Corollary~\ref{Th general max sum} that, when no assumptions are imposed to the directing random measure, the limit of the distribution of $n^{-1/2}\max_{1\leq k\leq n}S_k$}
\note{depends} on the conditional drift \note{and the conditional variance} of the random variables. In particular, we see that the probability of the random variables having negative drift makes a substantial contribution to this limit.

\medskip

The paper is organized as follows: In Section~\ref{section defs}, we fix notation and briefly \note{review} basic results of the theory of exchangeable random variables. Section~\ref{section results} is devoted to presenting \note{and proving the} different generalizations of the limit distribution of Erd\"os and Kac. \note{Finally, these results are furnished with examples in Section~\ref{Examples}}. 

\section{Definitions and auxiliary results}\label{section defs}
Let $\Pi(n)$ denote the set of permutations of $\{1,\ldots, n\}$. A sequence of random variables $\{X_n\}_{n\in\N}$ is said to be exchangeable if for any $n\in\N$, $X_1,\ldots, X_n$ are exchangeable, i.e., for any permutation $\pi\in\Pi(n)$,
\begin{equation*}
Law(X_1,\ldots,X_n) = Law(X_{\pi(1)},\ldots, X_{\pi(n)}).
\end{equation*}
Alternatively we write $(X_1,\ldots, X_n)\stackrel{d}{=}(X_{\pi(1)},\ldots, X_{\pi(n)})$. The concept of \note{exchangeability} was introduced by de Finetti in~\cite{DeF30}, who in particular proved that such a sequence is conditionally i.i.d. given the $\sigma$-field of permutable events.

\medskip

An essential tool in our proofs is de Finetti's theorem. Let $\f$ denote the collection of all \note{p.d.f.s} on $\mathbb R$ 
\note{with} the topology of weak convergence of distribution functions. De Finetti's theorem states that for an infinite sequence of exchangeable random variables $\{X_n\}_{n\in\N}$, there exists a unique probability measure $\mu$ on the Borel $\sigma$-field $\mathfrak{A}$ of subsets of $\f$ such that for any $n\geq 1$,
\begin{equation}\label{de_fin_thm}
\p\big(g(X_1,\ldots,X_n)\in B\big)=\int_{\f}\pF\big(g(X_1,\ldots,X_n)\in B\big)\mu(d F)
\end{equation}
\note{holds} for any Borel set $B\in\B(\note{\R})$ and any Borel function $g:\R^n\rightarrow\R$. Here, $\pF\big(g(X_1,\ldots,X_n)\in B\big)$ is the probability of the event under the assumption that the random variables $X_1,\ldots,X_n$ are independent with common \note{p.d.f.} $F$. The mean $\E_F g(X_1,\ldots,X_n)$ is obtained by integrating $g$ with respect to the \note{probability measure} $F$. \note{Let us} now denote by $\F\colon\Omega\to\f$ a random variable whose probability distribution is given by the measure $\mu$ from de Finetti's theorem. The conditional mean $\E_\F g(X_1,\ldots,X_n)$ is defined analogously to $\E_F g(X_1,\ldots,X_n)$ \note{and is itself} a random variable because the \note{p.d.f.} $\F$ is random. It should be noted that de Finetti's theorem fails for finite collections of exchangeable random variables. We refer to~\cite{Ald85} for further details on this subject.

\medskip

The law of large numbers (LLN) for exchangeable sequences was established by Hu and Taylor in~\cite{HT87}. They showed that for an exchangeable sequence $\{X_n\}_{n\in\N}$ such that $\E_\F|X_1|<\infty$ $\mu$-a.s.,
\begin{equation}\label{Hu_Taylor}
\frac{1}{n}S_n\xrightarrow{~\text{a.s.}~}0\quad\text{as}~n\to\infty\qquad\text{ if and only if } \qquad \E X_1X_2=0.
\end{equation}
It is not difficult to see that $\E X_1 X_2 = 0 $ is equivalent to $\E_\F X_1=0$ $\mu$-a.s. As already mentioned in the introduction, Blum, Chernoff, Rosenblatt and Teicher proved in~\cite{BCRT58} that for an exchangeable sequence with zero mean and variance one the CLT holds if and only if 
\begin{equation}\label{cond CLT ex seq}
\E X_iX_j=0\qquad\text{and}\qquad \E X_i^2 X_j^2=1\qquad\forall~i\neq j.
\end{equation}

In general, it is possible to obtain limit theorems for sums of exchangeable sequences under weaker assumptions. \note{De Finetti's theorem can be rephrased (see~\cite[Theorem 3.1]{Ald85}) by saying that the} infinite exchangeable sequence \note{$\{X_n\}_{n\in\N}$} is a mixture of i.i.d. \note{random variables} directed by \note{the} random measure \note{$\F$}, whose probability distribution $\mu$ is given in~\eqref{de_fin_thm}. With this notation, 
\[
\p_\F((X_1,\ldots, X_n)\in \note{A})=\prod_{i=1}^n\note{\F(A_i)},\qquad \note{A=A_1\times\cdots\times A_n}\in\B(\R^n),
\]
and \note{$\E_\F X_1=\int_\R x\,\F(dx)$}. Moreover, if $\E|X_1|<\infty$, then the LLN
\begin{equation}\label{LLN martingale}
\frac{1}{n}S_n\xrightarrow{~\text{a.s.}~}\E_\F X_1\qquad\text{as }~n\to\infty
\end{equation}
holds (see e.g.~\cite[p.17]{Ald85}). \note{This directly implies the necessity of~\eqref{Hu_Taylor} since $\E X_1X_2=0$ means that the directing random measure $\F$ has zero mean, i.e. $\E_\F X_1=0$ $\mu$-a.s.} Furthermore, we have that if $0<\E X_1^2<\infty$, then the CLT
\begin{equation}\label{CLT sums Aldous}
\frac{S_n-\note{n}\E_\F X_1}{\sqrt{n}\sigma_\F}\xrightarrow{\quad d\quad}\mathcal{N}(0,1)
\end{equation}
holds, where $\sigma_\F^2:=\E_\F (X_1-\E_{\note{\F}} X_1)^2$. \note{This formulation generalizes the necessity of~\eqref{cond CLT ex seq} because again $\F$ has zero mean and $\E X_1^2X_2^2=1$ is equivalent to the fact that $\F$ has variance one, i.e. $\sigma^2_\F=1$ a.s. These limit theorems can also be obtained in terms of conditional characteristic functions, see~\cite{YWL14}.}

\section{Limit theorem for maximum of sums of exchangeable random variables}\label{section results}
In this \note{section,} we investigate the limiting distribution of the \note{largest} partial sum of an exchangeable sequence of random variables. \note{In a first step, this limit is obtained under the assumption that the directing random measure $\F$ has zero mean and variance one. Secondly, the variance-one assumption is removed and the corresponding limiting theorem is derived. Finally, the latter result is applied to analyze the limit of the probability of the maximum of partial sums for a sequence with a general directing random measure.}

\medskip

\note{Let us} start by recalling the original result of Erd\"os and Kac. 

\begin{Th}\cite{EK46}\label{Th EK}
Let \note{$\{X_n\}_{n\in\N}$} be a sequence of i.i.d. random variables with zero mean and variance one, and let $S_k:=\sum_{i=1}^k X_i$. Then,
\[
\lim_{n\to\infty}\p(\max(S_1,\ldots, S_n)<x\sqrt{n})=G(x),
\]
where \note{$G\colon\R\to\R$ is given by}
\begin{equation}\label{lim EK}
G(x):=\note{(2\Phi(x)-1)\1_{[0,\infty)}(x)}
\end{equation}
\note{and $\Phi$ denotes the p.d.f. of the standard normal distribution.}
\end{Th}

A direct extension of this theorem in the exchangeable setting is obtained when we assume that the conditions for the \note{classical} CLT given in~\eqref{cond CLT ex seq} are satisfied.

\begin{Prop}\label{Th EK ex case}
Let $\{X_n\}_{n\in\N}$ be an exchangeable sequence of random variables with zero mean and variance one satisfying~\eqref{cond CLT ex seq}. Then,
\[
\lim_{n\to\infty}\p(\max(S_1,\ldots, S_n)<x\sqrt{n})=G(x),
\]
where \note{$G\colon\R\to\R$} is given by~\eqref{lim EK}.
\end{Prop}

\begin{proof}
Since $\p_F(\max(S_1,\ldots ,S_n)<x\sqrt{n})$ is uniformly bounded by one and $\mu$ is a probability measure, applying de Finetti's theorem, Lebesgue dominated convergence theorem and Theorem~\ref{Th EK} to the conditional probability yields
\begin{multline*}
\lim_{n\to\infty}\p(\max(S_1,\ldots, S_n)<x\sqrt{n})\\
=\int_\f\lim_{n\to\infty}\p_F(\max(S_1,\ldots, S_n)<x\sqrt{n})\,\mu(dF)=G(x).
\end{multline*}
\end{proof}

\note{
\begin{Rem}
Notice that in fact, every limiting result originally proved by Erd\"os and Kac in~\cite{EK46} can be obtained in the same fashion.
\end{Rem}
}

\note{The next natural step to generalize Proposition~\ref{Th EK ex case} consists in considering an exchangeable sequence $\{X_n\}_{n\in\N}$ whose directing random measure $\F$ only satisfies the zero-mean condition. From~\eqref{CLT sums Aldous} we know that} 
\begin{equation}\label{CLT zero mean}
\frac{1}{\sqrt{n}}S_n\xrightarrow{\; d\;} Z\cdot\sigma_\F,
\end{equation}
where $Z\sim \mathcal{N}(0,1)$ is independent of $\sigma_\F$ \note{and $\F$ has distribution measure $\mu$ from de Finetti's theorem}. \note{Let us now define $G_\mu\colon\R\to\R$ as
\begin{equation}\label{limit distr}
G_\mu(x):=\int_\f \1_{(0,\infty)}\,(\sigma^2_F)G(x/\sigma_F)\,\mu(dF),
\end{equation}
where $G$ was given~\eqref{lim EK}. Then we have the following result.}

\begin{Th}\label{Th EK weak}
Let $\{X_n\}_{n\in\N}$ be an exchangeable sequence of random variables with zero mean and variance $0<\E X_1^2<\infty$ such that $\E X_1X_2=0$. Then,
\begin{equation}\label{lim EK weak}
\lim_{n\to\infty}\p(\max(S_1,\ldots,S_n)<x\sqrt{n})=\p(\sigma^2_\F=0)\1_{[0,\infty)}(x)+G\note{_\mu}(x),
\end{equation}
where $G_\mu\colon\R\to\R$ is given in~\eqref{limit distr} and $\mu$ is the distribution of the directing random measure of the sequence $\{X_n\}_{n\in\N}$.
\end{Th}

\begin{proof}
In view of~\eqref{CLT zero mean}, de Finetti's theorem, Lebesgue dominated convergence theorem and Theorem~\ref{Th EK} lead to
\begin{align*}
&\lim_{n\to\infty}\p(\max(S_1,\ldots,S_n)<x\sqrt{n})\\
&=\p(\sigma^2_\F=0)\1_{[0,\infty)}(x)+
\int_\f\lim_{n\to\infty}\1_{(0,\infty)}(\sigma_F^2)\p_F(\max(S_1,\ldots,S_n)<x\sqrt{n})\,\mu(dF)\\
&=\p(\sigma^2_\F=0)\1_{[0,\infty)}(x)+\int_\f\1_{(0,\infty)}(\sigma_F^2) G(x/\sigma_F)\,\mu(dF).
\end{align*}
\end{proof}
\note{
\begin{Rem}\label{remark degenerate}
Notice that in the exchangeable setting one may encounter sequences of non constant random variables with $\p(\sigma_\F^2=0)>0$ (see Example~\ref{Example 2}). 
In particular, Theorem~\ref{Th EK weak} shows that if the sequence is non-degenerated in the sense that the conditional variance is almost surely positive, i.e. $\p(\sigma_\F^2>0)=1$, then the limiting distribution in~\eqref{lim EK weak} becomes the mixture
\begin{equation}\label{eq mixture}
\int_\f G(x/\sigma_F)\,\mu(dF).
\end{equation}
\end{Rem} 
\begin{Rem}
The exchangeable counterparts of the limiting distributions in~\cite{EK46} under the assumptions of Theorem~\ref{Th EK weak} can be derived in the same fashion.
\end{Rem}
}

\medskip

\note{We conclude this section applying Theorem~\ref{Th EK weak} to investigate the limit of the distribution of $n^{-1/2}\max_{1\leq k\leq n}S_k$ when no assumptions are imposed to the directing random measure of the exchangeable sequence $\{X_n\}_{n\in\N}$.}

\begin{Cor}\label{Th general max sum}
Let $\{X_n\}_{n\in\N}$ be an exchangeable sequence of random variables with zero mean and variance $0<\E X_1^2<\infty$.
Then,
\begin{multline*}
\lim_{n\to\infty}\p(\max(S_1,\ldots,S_n)<x\sqrt{n})\\
=\p(\E_\F X_1<0)+\p(\E_\F X_1=0,\sigma^2_\F=0)\1_{[0,\infty)}(x)+G'_\mu(x),
\end{multline*}
where $G'\colon\R\to\R$ is given by
\begin{equation*}
G'_\mu(x)=\int_\f\1_{\{0\}}(\E_F X_1)\1_{(0,\infty)}(\sigma_F^2)\, G(x/\sigma_F)\,\mu(dF)
\end{equation*}
and $G$ is defined in~\eqref{lim EK}.
\end{Cor}

\begin{proof}
By de Finetti's theorem,
\begin{align*}
\p&(\max(S_1,\ldots,S_n)<x\sqrt{n})=\int_{\f_+}\p_F(\max_{1\leq k\leq n}S_k <x\sqrt{n})\,\mu(dF)\\
&+\int_{\f_0}\p_F(\max_{1\leq k\leq n}S_k <x\sqrt{n})\,\mu(dF)+\int_{\f
_-}\p_F(\max_{1\leq k\leq n}S_k <x\sqrt{n})\,\mu(dF)\\
&=:I_{n,+}(x)+I_{n,0}(x)+I_{n,-}(x),
\end{align*}
where $\f_+:=\{F\in\f\,|\,\EF X_1>0\}$, $\f_0:=\{F\in\f\,|\,\EF X_1=0\}$, and $\f_-:=\{F\in\f\,|\,\EF X_1<0\}$.

On the one hand, for any $F\in\f_+$,
\[
\p_F(\max_{1\leq k\leq n}S_k <x\sqrt{n})\leq\p_F(S_n-n\E_F X_1<x\sqrt{n})\leq\frac{\sigma_F^2}{n x^2}
\]
which tends to zero as $n\to\infty$. Lebesgue dominated convergence theorem thus yields $\lim\limits_{n\to\infty}I_{n,+}(x)=0$. On the other hand, following the proof of Theorem~\ref{Th EK weak} we have that
\begin{align*}
\lim_{n\to\infty} I_{n,0}(x)&=\int_{\f_0}\1_{(0,\infty)}(\sigma_F)\1_{[0,\infty)}(x)\,\mu(dF)+\int_{\f_0}\1_{(0,\infty)}(\sigma_F)\,G(x/\sigma_F)\,\mu(dF)\\
&=\p(\E_\F X_1=0,\sigma^2_\F=0)\1_{[0,\infty)}(x)+G'_\mu(x).
\end{align*}

Finally, if $F\in\f_-$, the H\'ayek-R\'enyi inequality leads to
\[
\p_F(\max_{1\leq k\leq n}S_k <x\sqrt{n})\geq1-\sum_{k=1}^n\frac{\sigma_F^2}{(x\sqrt{n}-k\E_F X_1)^2}\geq 1-\frac{\sigma_F^2}{\E_F X_1 x\sqrt{n}},
\]
which tends to one as $n\to\infty$. By Lebesgue dominated convergence theorem, $\lim\limits_{n\to\infty}I_{n,-}(x)=\p(\E_\F X_1<0)$ and the result follows.
\end{proof}
\note{
\begin{Rem}
Notice that the limit appearing in Corollary~\ref{Th general max sum} is not a distribution function unless the assumptions reduce to those of Theorem~\ref{Th EK weak}. Thus, convergence in distribution holds in Proposition~\ref{Th EK ex case} and Theorem~\ref{Th EK weak}, but not in the general situation of Corollary~\ref{Th general max sum}. 
\end{Rem}
}
\section{Examples}\label{Examples}
We finish our discussion with some examples that furnish the results presented in the previous section.

\begin{Ex}\label{Example 1}
Let $\{Y_n\}_{n\in\N}$ be an exchangeable sequence of random variables with $\E Y_1 Y_2=0$ and $\E Y_1^2Y^2_2=1$ that take values in $\{-1,1\}$, and let $\{Z_n\}_{n\in\N}$ be a sequence of i.i.d. standard normal distributed random variables independent of $\{Y_n\}_{n\in\N}$. The sequence 
\[
\{X_n\}_{n\in\N}:=\{Y_n+Z_n\}_{n\in\N}
\]
is exchangeable and the process $S_n=:\sum_{k=1}^nX_k$ may be called exchangeable random walk plus noise. This model appears for instance in Bayesian dynamic modeling~\cite[Chapter 8]{DDPS13}. For this sequence, $\E X_n=\E Y_n+\E Z_n=0$ and 
\[
\E X_1 X_2 =\E(Y_1+Z_1)(Y_2+Z_2)=\E Y_1Y_2+\E Y_1\E Z_2+\E Z_1 \E Y_2 +\E Z_1\E Z_2=0.
\]
Moreover, it is non-degenerate in the sense that
\[
\p(\sigma_\F^2>0)=\p(\E_\F X_1^2>0)\geq\p(\E_\F Y_1^2>0)=1,
\]
and therefore $\p(\sigma_\F^2=0)=0$. Thus, we can investigate the asymptotic distribution of the stopping times
\[
T_n(x):=\inf\{k\geq 1~\colon~S_k>x\sqrt{n}\}
\]
by studying the asymptotic distribution of the maximum of partial sums $S_k$. Applying Theorem~\ref{Th EK weak} we get
\begin{align*}
\lim_{n\to\infty}\p(T_n(x)\leq n)&=\lim_{n\to\infty}\p(\max(S_1,\ldots,S_n)> x\sqrt{n})\\
&=1-\lim_{n\to\infty}\p(\max(S_1,\ldots,S_n)\leq x\sqrt{n})\\
&=1-G_\mu(x),
\end{align*}
where $G_\mu$ is given by~\eqref{eq mixture}.
%
\end{Ex}

As pointed out in Remark~\ref{remark degenerate}, it is possible to have an exchangeable sequence of non constant random variables with $\p(\sigma^2_\F=0)>0$. The following example illustrates how this situation may arise in applications.
\begin{Ex}\label{Example 2}
In financial modelling, the risk of a financial asset can be represented by a sequence $\{\xi_n \}_{n\in\N}$ of i.i.d. real-valued random variables, for instance with zero mean and positive variance $\sigma^2>0$. The quantity $\max_{1\leq k\leq n}\sum_{i=1}^k\xi_i$ thus expresses the maximal loss over a certain period. Exchangeable models appear in this setting for instance by introducing an independent default indicator $Y$ 
such that $\p(Y=1)=1-\p(Y=0)=p$, $p\in(0,1)$. In this case, the sequence
\[
\{X_n\}_{n\in\N}:=\{Y\xi_n\}_{n\in\N}
\]
is exchangeable with $\E X_1=p\,\E\xi_1=0$ and $\E X_1^2=p\,\E\xi_1^2=p\,\sigma^2´>0$. This example is especially illustrative because it is possible to verify Theorem~\ref{Th EK weak} by calculating the limiting probability of the maximal loss directly. Since $\xi_n$ are i.i.d. we have that $\E X_1X_2=\E Y^2\xi_1\xi_2=p\,\E\xi_1\xi_2=p(\E\xi_1)^2=0$ and $\E X^2_1X^2_2=\E Y^4\xi^2_1\xi^2_2=p\,\E\xi^2_1\xi^2_2>0$, which might not necessarily be one.

For the partial sums $S_n:=\sum_{i=1}^n X_i$, the limiting probability
\[
\lim_{n\to\infty}\p(\max(S_1,\ldots ,S_n) \leq x\sqrt{n})
\]
can be obtained as follows: Let $\tilde{S}_n:=\sum_{k=1}^n \xi_k$. By conditioning on $Y$ and applying the classical result of Erd\"os and Kac to $\tilde{S}_n/\sigma$ we have that 
\begin{align*}
\lim_{n\to\infty}\p(\max_{1\leq k\leq n}S_k \leq x\sqrt{n})
&=p\,G(x/\sigma)+(1-p)\1_{[0,\infty)}(x).
\end{align*}

Let us now see that this coincides with the limiting expression in our Theorem. In this case, the directing measure $\F$ is discrete and takes the values $F_1$ and $F_2$ with probability $p$ and $1-p$ respectively, and 
\[
X_n=\left\{
\begin{array}{ll}
\xi_n&\text{ under }F_1,\\
&\\
0&\text{ under }F_2,
\end{array}
\right.
\]
for each $n\in\N$. Thus, $\E_{F_1} X_1=\E\xi_1=0=\E_{F_2}X_2$ and 
$\Var_{F_1}X_1=\sigma^2$, $\Var_{F_2}X_1=0$. In particular $X_1\equiv 0$ under $F_2$, hence
\[
\p(\E_\F X_1=0,\sigma_\F^2=0)=\p(\F=F_2)=1-p
\]
and
\[
G_\mu(x)=\int_\f\1_{\{0\}}(\E_F X_1)\1_{(0,\infty)}(\sigma_F^2)\, G(x/\sigma_F)\,\mu(dF)=p\,G(x/\sigma)
\]
as desired.
\end{Ex}
%

\section*{Acknowledgments}
This research was partially supported by the RFBR grant 13-01-00612. The authors thank Prof. Bulinski, Prof. Spodarev and the referees for helpful comments.

\bibliographystyle{amsplain}
\bibliography{References}

\providecommand{\bysame}{\leavevmode\hbox to3em{\hrulefill}\thinspace}
\providecommand{\MR}{\relax\ifhmode\unskip\space\fi MR }
\providecommand{\MRhref}[2]{%
  \href{http://www.ams.org/mathscinet-getitem?mr=#1}{#2}
}
\providecommand{\href}[2]{#2}
\begin{thebibliography}{10}

\bibitem{Ald85}
D.~J. Aldous, \emph{Exchangeability and related topics}, \'{E}cole d'\'et\'e de
  probabilit\'es de {S}aint-{F}lour, {XIII}---1983, Lecture Notes in Math.,
  vol. 1117, Springer, Berlin, 1985, pp.~1--198.

\bibitem{Ald10}
\bysame, \emph{Exchangeability and continuum limits of discrete random
  structures}, Proceedings of the {I}nternational {C}ongress of
  {M}athematicians. {V}olume {I}, Hindustan Book Agency, New Delhi, 2010,
  pp.~141--153.

\bibitem{Aus08}
T.~Austin, \emph{On exchangeable random variables and the statistics of large
  graphs and hypergraphs}, Probab. Surv. \textbf{5} (2008), 80--145.

\bibitem{Ber62}
S.~M. Berman, \emph{Limiting distribution of the maximum term in sequences of
  dependent random variables}, Ann. Math. Statist. \textbf{33} (1962),
  894--908.

\bibitem{BR97}
P.~Berti and P.~Rigo, \emph{A {G}livenko-{C}antelli theorem for exchangeable
  random variables}, Statist. Probab. Lett. \textbf{32} (1997), no.~4,
  385--391.

\bibitem{BCRT58}
J.~R. Blum, H.~Chernoff, M.~Rosenblatt, and H.~Teicher, \emph{Central limit
  theorems for interchangeable processes}, Canad. J. Math. \textbf{10} (1958),
  222--229.

\bibitem{Bui82}
T.~A. Buishand, \emph{Some methods for testing the homogeneity of rainfall
  records}, Journal of Hydrology \textbf{58} (1982), no.~1, 11 -- 27.

\bibitem{BR15}
A.~Bulinski and A.~Rakitko, \emph{Simulation and analytical approach to the
  identification of significant factors}, Communications in Statistics Part B:
  Simulation and Computation \textbf{44} (2015), 1--23.

\bibitem{Cha06}
S.~Chatterjee, \emph{A generalization of the {L}indeberg principle}, Ann.
  Probab. \textbf{34} (2006), no.~6, 2061--2076.

\bibitem{CM15}
A.~Coen and R.~H Mena, \emph{Ruin probabilities for {B}ayesian exchangeable
  claims processes}, Journal of Statistical Planning and Inference (2015), in
  press.

\bibitem{DDPS13}
P.~Damien, P.~Dellaportas, N.~G. Polson, and D.~A. Stephens (eds.),
  \emph{Bayesian theory and applications}, Oxford University Press, Oxford,
  2013.

\bibitem{DeF30}
B.~{De Finetti}, \emph{Funzione caratteristica di un fenomeno aleatorio}, Atti
  Accad. Naz. Lincei Rend. Cl. Sci. Fis. Mat. Nat. \textbf{4} (1930), 86--133.

\bibitem{Eag82}
G.~K. Eagleson, \emph{Weak limit theorems for exchangeable random variables},
  Exchangeability in probability and statistics ({R}ome, 1981), North-Holland,
  Amsterdam-New York, 1982, pp.~251--268.

\bibitem{EK46}
P.~Erd{\"o}s and M.~Kac, \emph{On certain limit theorems of the theory of
  probability}, Bull. Amer. Math. Soc. \textbf{52} (1946), 292--302.

\bibitem{FLR12}
S.~Fortini, L.~Ladelli, and E.~Regazzini, \emph{Central limit theorem with
  exchangeable summands and mixtures of stable laws as limits}, Boll. Unione
  Mat. Ital. (9) \textbf{5} (2012), no.~3, 515--542.

\bibitem{GST00}
A.~Gerardi, F.~Spizzichino, and B.~Torti, \emph{Exchangeable mixture models for
  lifetimes: the role of ``occupation numbers''}, Statist. Probab. Lett.
  \textbf{49} (2000), no.~4, 365--375.

\bibitem{Hey67}
C.~C. Heyde, \emph{Asymptotic renewal results for a natural generalization of
  classical renewal theory}, J. Roy. Statist. Soc. Ser. B \textbf{29} (1967),
  141--150.

\bibitem{Hur51}
H.~E. Hurst, \emph{Long-tem storage capacity of reservoirs}, Trans. Am. Soc.
  Eng. \textbf{116} (1951), no.~1, 770--799.

\bibitem{JJS87}
B.~James, K.~L. James, and D.~Siegmund, \emph{Tests for a change-point},
  Biometrika \textbf{74} (1987), no.~1, 71--83.

\bibitem{Kin80}
J.~F.~C. Kingman, \emph{Mathematics of genetic diversity}, CBMS-NSF Regional
  Conference Series in Applied Mathematics, vol.~34, Society for Industrial and
  Applied Mathematics (SIAM), Philadelphia, Pa., 1980.

\bibitem{Nov12}
S.~Y. Novak, \emph{Extreme value methods with applications to finance},
  Monographs on Statistics and Applied Probability, vol. 122, CRC Press, Boca
  Raton, FL, 2012.

\bibitem{PR09}
B.~L.~S. Prakasa~Rao, \emph{Conditional independence, conditional mixing and
  conditional association}, Ann. Inst. Statist. Math. \textbf{61} (2009),
  no.~2, 441--460.

\bibitem{Sie68}
D.~Siegmund, \emph{On the asymptotic normality of one-sided stopping rules},
  Ann. Math. Statist \textbf{39} (1968), 1493--1497.

\bibitem{HT87}
R.~L. Taylor and T.~C. Hu, \emph{On laws of large numbers for exchangeable
  random variables}, Stochastic Anal. Appl. \textbf{5} (1987), no.~3, 323--334.

\bibitem{Web80}
N.~C. Weber, \emph{A martingale approach to central limit theorems for
  exchangeable random variables}, J. Appl. Probab. \textbf{17} (1980), no.~3,
  662--673.

\bibitem{YZ08}
C.~Yu and D.~Zelterman, \emph{Sums of exchangeable {B}ernoulli random variables
  for family and litter frequency data}, Computational Statistics and Data
  Analysis \textbf{52} (2008), no.~3, 1636 -- 1649.

\bibitem{YL15}
D.~Yuan and S.~Li, \emph{Extensions of several classical results for
  independent and identically distributed random variables to conditional
  cases}, J. Korean Math. Soc. \textbf{52} (2015), no.~2, 431--445.

\bibitem{YWL14}
D.~Yuan, L.~Wei, and L.~Lei, \emph{Conditional central limit theorems for a
  sequence of conditional independent random variables}, J. Korean Math. Soc.
  \textbf{51} (2014), no.~1, 1--15.

\end{thebibliography}

\end{document}